\documentclass{amsart}

\usepackage{amsmath, amssymb, amsthm, amsfonts}

\usepackage{adjustbox}
\usepackage{graphicx} %Include figure files
\usepackage{subcaption}
\usepackage{enumerate}
\usepackage{epsfig}
\usepackage{url}
\usepackage{appendix}

\newtheorem{theorem}{Theorem}

\newtheorem{proposition}{Proposition}

\theoremstyle{definition}
\newtheorem{definition}{Definition}

\begin{document}

\title{Insulated primes}

\author{Abhimanyu Kumar}
\address{Department of Electrical and Instrumentation Engineering, Thapar Institute of Engineering and Department, Patiala, Punjab 147004, India}
\email{akumar6\_be17@thapar.edu}

\author{Anuraag Saxena}
\address{Department of Electronics and Communication Engineering, Thapar Institute of Engineering and Technology, Patiala, Punjab 147004, India}
\email{asaxena1\_be19@thapar.edu}

\begin{abstract}
The degree of insulation of a prime $p$ is defined as the largest interval around it in which no other prime exists. Based on this, the $n$-th prime $p_{n}$ is said to be insulated if and only if its degree of insulation is higher than its neighboring primes. Consequently, a new special sequence emerges given as 7, 13, 23, 37, 53, 67, 89, 103, 113, 131, 139, 157, 173, 181, 193, 211, 233, 277, 293, and so on. This paper presents several properties and intriguing relations concerning degree of insulation and insulated primes. Finally, the reader is left with a captivating open problem.
\end{abstract}


\maketitle

\section{Introduction}

%For instance, primes that are present in a specific form such as Proth prime, Mersenne prime, and Fermat prime \cite{Robinson1954}. Sequences of primes are also defined if they obey certain property such as Chen prime \cite{Chen1966}, Lucky prime \cite{Hawkins1957}, Wilson prime \cite{Goldberg1953}, Ramanujan prime \cite{Sondow2009}, etc. which have captured attention due to associated unsolved problems. 
%\cite{Caldwell_Mersenne} \cite{Caldwell_Proth}

%, such as twin-primes (A077800), Wilson primes (A007540), isolated primes (A023188), etc. 
%As yet there are only three Wilson primes known, but it has been conjectured that infinitely many Wilson primes exist, which is still an open problem. Researchers are also interested in Wilson prime of order $n$ is a prime $p$ such that $p^2$ divides $(n-1)!(p-n)! - (-1)^n$

Prime numbers and their several special sub-sequences have continuously fascinated both young enthusiasts and experienced researchers \cite{Guy2004}. During the covid-19 lockdown, two new sequences A339270 and A339148 were introduced in OEIS, namely degree of insulation and insulated primes respectively \cite{Kumar2020}.

\begin{definition}
Degree of insulation $D: \mathbb{P} \to \mathbb{N}$ of a prime $p$ is defined as the maximum of the set $X_p = \{m:\pi(p-m)=\pi(p+m)-1, m\in\mathbb{N}\}$, where $\pi(n)$ is the prime counting function.
\end{definition}

Since $D(p)$ can be interpreted as the largest spread around $p$ containing only the prime $p$, so, any deterministic procedure to evaluate $D(p)$ is required to either compute prime counting function $\pi(x)$ or to determine the surrounding primes $(p_{n-1},p_{n+1})$ of the given prime $p_n$. The plot of $D(p)$ values for primes less than $1000$ is shown in Figure \ref{Dp1000}. 

\begin{figure}[h!]
\centering
\includegraphics[scale=0.45]{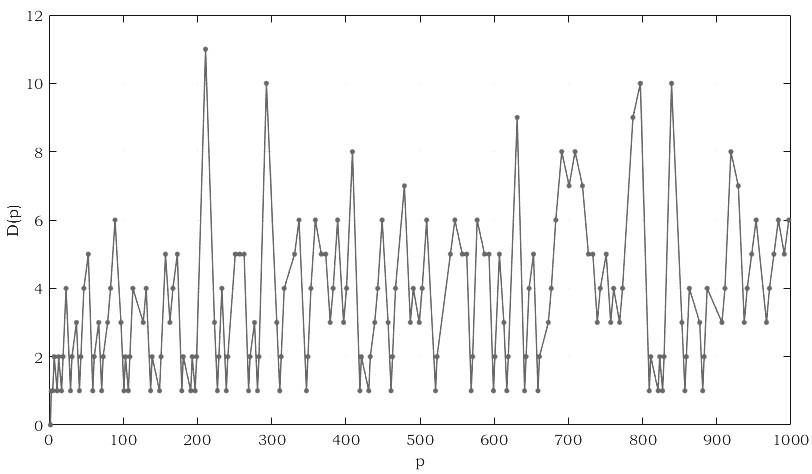}
\caption{Plot of $D(p)$ for primes less than $1000$}
\label{Dp1000}
\end{figure}

Consider the prime triplet $(p_{n-1},p_{n},p_{n+1})=(19, 23, 29)$, then degree of insulation for prime $p_{n-1} = 19$ is calculated as follows: 
\begin{align*}
\pi(19-1) &\overset{\mathrm{?}}{=} \pi(19+1)-1 \Rightarrow 7\overset{\mathrm{?}}{=}8-1 \Rightarrow 7 = 7 \\
\pi(19-2) &\overset{\mathrm{?}}{=} \pi(19+2)-1 \Rightarrow 7\overset{\mathrm{?}}{=}8-1 \Rightarrow 7 = 7 \\
\pi(19-3) &\overset{\mathrm{?}}{=} \pi(19+3) -1 \Rightarrow 7\overset{\mathrm{?}}{=}9-1 \Rightarrow 7 \neq  8
\end{align*}
which gives $D(19)=2$. This process highlights two key results: (a) if $\alpha\notin X_p$ then $(\alpha+r)\notin X_p$ for all $r\geq 0$, and (b) if $\alpha\notin X_p$ then $D(p)<\alpha$ for prime $p$. On similar lines as illustarted above, one can evaluate $D(23)=4$ and $D(29)=1$. Importantly, the observation that $D(23)>D(19)$ and $D(23)>D(29)$ gives rise to the idea of insulated primes which is formally defined below.

\begin{definition}
The $n$-th prime $p_{n}$ is said to be insulated if and only if $D(p_{n}) > \max\{D(p_{n-1}), D(p_{n+1})\}$.
\end{definition}

\begin{figure}[h!]
\centering
\includegraphics[scale=0.45]{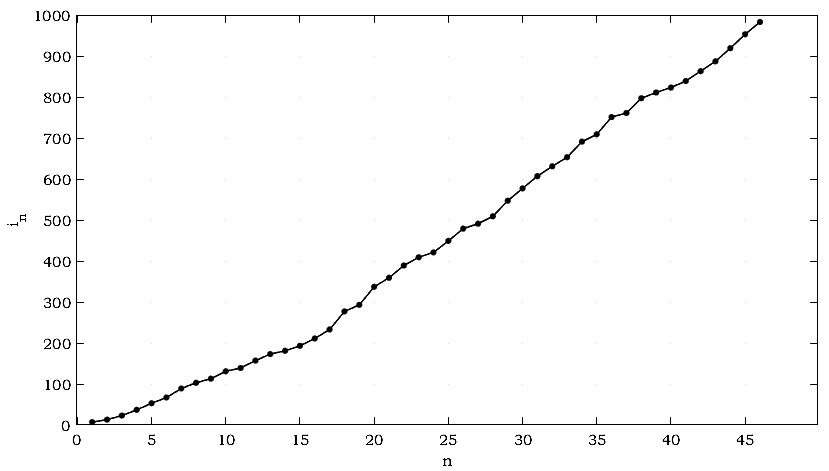}
\caption{Plot of $i_n$ versus $n$ for primes less than $1000$}
\label{in1000}
\end{figure}

Figure \ref{in1000} shows the plot of $n$-th insulated prime $i_n$ vs $n$. Some quick observations regarding insulated primes\footnote{Since $D(p)$ essentially insulates $p$ from neighboring primes, so the name ``insulated primes" is given. Initially, ``isolated" was intended to be used, but it is already taken (OEIS: A023188).} are: (a) primes just adjacent to an insulated prime can never be insulated, (b) $i_n$ seems to obey a linear-like fit. 

With the complete motivation behind this concept laid out, the subsequent sections shall investigate it further. %Also, note that the set of insulated primes has certain common elements with the set of isolated primes (OEIS: A007510). 
The remainder of the paper focuses on analytic and heuristic analysis of degree of insulation and insulated primes.

\section{Main analysis and results}

\begin{proposition}
Primes just adjacent to an insulated prime can never be insulated.
\end{proposition}
\begin{proof}
Let $p_n$ be an insulated prime, then $D(p_{n-1})<D(p_{n})$ and $D(p_{n+1})<D(p_{n})$ by definition. Now, for $p_{n-1}$ to be insulated prime, the conditions $D(p_{n-2})<D(p_{n-1})$ and $D(p_{n})<D(p_{n-1})$ must hold. Clearly, the latter condition is contradictory to the condition for $p_n$, therefore, $p_{n-1}$ cannot an insulated prime. For $p_{n+1}$ to be an insulated prime, the conditions $D(p_{n})<D(p_{n+1})$ and $D(p_{n+2})<D(p_{n+1})$ must hold. But in this case, the prior condition would not be feasible, therefore, $p_{n+1}$ is also not an insulated prime. Hence, proved.
\end{proof}

%\begin{proposition}
%If $\alpha\notin X_p$ then $(\alpha+r)\notin X_p$ for all $r\geq 0$.
%\end{proposition}
%\begin{proof}
%For some $r\geq 0$, we have $\pi(p+\alpha+r)\geq\pi(p+\alpha)$ and $\pi(p-\alpha)\geq\pi(p-\alpha-r)$ because $\pi(n)$ is an increasing function. This implies $\pi(p+\alpha+r)-\pi(p-\alpha-r)\geq \pi(p+\alpha)-\pi(p-\alpha)$. As $X_p$ contains all the possible candidates for being $D(p)$ and $\alpha\notin X_p$, then $\pi(p+\alpha)-\pi(p-\alpha)> 1$ because there exists at least one prime $p$ such that $p-\alpha\leq p \leq p+\alpha$. Using this gives $\pi(p+\alpha+r)-\pi(p-\alpha-r)\geq \pi(p+\alpha)-\pi(p-\alpha) > 1$ which implies $\pi(p+\alpha+r)-\pi(p-\alpha-r)\neq 1$. Thus, $\alpha+r$ is also not a possible candidate for $D(p)$.
%\end{proof}
%
%\begin{corollary} \label{main}
%For prime $p$, if $\alpha\notin X_p$ then $D(p)<\alpha$.
%\end{corollary}
%\begin{proof}
%As $(\alpha+r) \notin X_p$ that is $\alpha+r$ is not a candidate for being $D(p)$ for all $r\geq 0$. Now, if non-zero $D(p)$ exists then $D(p)$ must be less than $\alpha$.
%\end{proof}

\begin{proposition} \label{Xp}
For prime $p$, if $\alpha\notin X_p$ then $D(p)<\alpha$.
\end{proposition}
\begin{proof}
Since $\pi(n)$ is an increasing function, so for every $r\geq 0$, we have $\pi(p+\alpha+r)\geq\pi(p+\alpha)$ and $\pi(p-\alpha)\geq\pi(p-\alpha-r)$ which combines to give $\pi(p+\alpha+r)-\pi(p-\alpha-r)\geq \pi(p+\alpha)-\pi(p-\alpha)$. As $X_p$ contains all the possible candidates for being $D(p)$ and $\alpha\notin X_p$, then $\pi(p+\alpha)-\pi(p-\alpha)> 1$ because there exists at least one prime $p$ such that $p-\alpha\leq p \leq p+\alpha$. One gets $\pi(p+\alpha+r)-\pi(p-\alpha-r)\geq \pi(p+\alpha)-\pi(p-\alpha) > 1$ which implies $\pi(p+\alpha+r)-\pi(p-\alpha-r)\neq 1$. This concludes that $\alpha+r$ is also not a possible candidate for $D(p)$, thus, $D(p)$ must be less than $\alpha$.
\end{proof}

\begin{theorem}
For a prime $p$, every $m\in X_p$ obeys
\begin{equation} \label{check}
\log(p+m)\log(p-m) < m\log\left((p-m)^c (p+m)\right) + p\log\left(\frac{(p-m)^c}{p+m}\right)
\end{equation}
provided $p\geq 17$, where $c=30\frac{\log{113}}{113}$.
\end{theorem}
\begin{proof}
From \cite{Rosser1962}, we have
\begin{equation*}
\frac{x}{\log{x}} < \pi(x) < \frac{cx}{\log{x}}
\end{equation*}
which gives
\begin{equation*}
\pi(p+m) - \pi(p-m) < \frac{c(p+m)}{\log(p+m)} - \frac{p-m}{\log(p-m)} .
\end{equation*}
For $m$ to belong to set $X_p$, the left hand side of the inequality must be one. On solving the expression and rearranging, we get the desired result.
\end{proof}

\begin{proposition}
For primes $p_n\geq 3$, if $D(p_n)=1$ then $D(p_{n+1})=2$.
\end{proposition}
\begin{proof}
The given $D(p_n)=1$ implies 
\begin{align*}
\pi(p_{n}-1) = \pi(p_{n}+1)-1 \\
\pi(p_{n}-2) \neq \pi(p_{n}+2)-1 ,
\end{align*}
and remain negated thereafter. In view of Proposition \ref{Xp}, only first two conditions are relevant for subsequent analysis. Assuming $p_n-2$ is a prime, one finds that
\begin{align*}
\pi(p_{n}-1) &\overset{\mathrm{?}}{=} \pi(p_{n}+1)-1 \Rightarrow n-1 = n-1 \\
\pi(p_{n}-2) &\overset{\mathrm{?}}{=} \pi(p_{n}+2)-1 \Rightarrow n-1= n-1
\end{align*}
which is a contradiction since the second condition should have been a non-equality, so $p_n-2$ is not a prime. Therefore, clearly $p_n+2$ must be the prime (which may be checked likewise); hence, $p_{n+1}=p_n+2$. Now, to find $D(p_{n+1})$:
\begin{align*}
\pi(p_{n+1}-1) &\overset{\mathrm{?}}{=} \pi(p_{n+1}+1)-1 \Rightarrow n\overset{\mathrm{?}}{=}(n+1)-1 \Rightarrow n = n \\
\pi(p_{n+1}-2) &\overset{\mathrm{?}}{=} \pi(p_{n+1}+2)-1 \Rightarrow n\overset{\mathrm{?}}{=}(n+1)-1 \Rightarrow n = n \\
\pi(p_{n+1}-3) &\overset{\mathrm{?}}{=} \pi(p_{n+1}+3)-1 \Rightarrow n-1\overset{\mathrm{?}}{=}(n+1)-1 \Rightarrow n-1 \neq n
\end{align*}
which shows $D(p_{n+1})=2$.
\end{proof}

\begin{theorem} \label{Dg}
We have $D(p_k) = \min\{p_{k+1}-p_k-1, p_k-p_{k-1}\}$.
\end{theorem}
\begin{proof}
Let $\lfloor . \rfloor_{\mathbb{P}}$ and $\lceil . \rceil_{\mathbb{P}}$ be the prime floor and prime ceiling functions \cite{Kumar2022}. Degree of insulation $D(p) = \max\{m:\pi(p-m)=\pi(p+m)-1, m\in\mathbb{N}\}$ can be equivalently expressed as $D(p) = \max\{m: \lfloor{p+m}\rfloor_{\mathbb{P}}=\lceil{p-m}\rceil_{\mathbb{P}}, m\in\mathbb{N}\}$. This shows $D(p_k)$ is the largest $m$ such that there is no prime except $p_k$ from $p_k-m+1$ to $p_k+m$. Hence, we have $D(p_k) = \min\{p_{k+1}-p_k-1, p_k-p_{k-1}\} = \min\{\lceil{p+1}\rceil_{\mathbb{P}}-p-1, p-\lfloor{p-1}\rfloor_{\mathbb{P}}\}$. 
\end{proof}
%The first moment when $\lfloor{p+m}\rfloor_{\mathbb{P}}-\lceil{p-m}\rceil_{\mathbb{P}}=1$ for some $m$, then $D(p)=m-1$. It may be expressed as $\lfloor{p+D(p)+1}\rfloor_{\mathbb{P}}-\lceil{p-D(p)-1}\rceil_{\mathbb{P}}=1$. 

\begin{proposition} \label{D<}
There exists a constant $\theta$ such that $D(p_k) < p_k^\theta - 1$ for sufficiently large $k$.
\end{proposition}
\begin{proof}
Hoheisel \cite{Hoheisel1930} showed that there exists a constant $\theta < 1$ such that
\begin{align*}
\pi(x+x^{\theta}) - \pi(x) \sim \frac{x^{\theta}}{\log(x)} \quad \text{as} ~ x\to \infty,
\end{align*}
hence, showing $p_{n+1} - p_n < p_n^\theta$ for large $n$. Using it with Theorem \ref{Dg} gives
\begin{align*}
D(p_k) &< \min\{p_k^\theta-1, p_{k-1}^\theta\} \\
&= p_k^\theta \min\left\{1-\frac{1}{p_k^\theta}, \left(\frac{p_{k-1}}{p_k}\right)^\theta\right\} \\
&< p_k^\theta \min\left\{1-\frac{1}{p_k^\theta}, 1\right\} \\
&= p_k^\theta - 1
\end{align*}
which is the desired result.
\end{proof}

The constant $\theta$ has been extensively studied in literature (for instance, see \cite{Alweiss2018}). Hoheisel obtained the possible value $\frac{32999}{33000}$, which was subsequently improved to $\frac{249}{250}$ by Heilbronn \cite{Heilbronn1933}. Thereafter, its value has been substantially reduced \cite{Chudakov1936, Huxley1972, Heath-Brown1979, Heath-Brown1988}. In 2001, Baker-Harman-Pintz \cite{Baker2001} obtained that $\theta$ may be taken to $0.525$ which is the best known unconditional result. Under the assumption that the Riemann hypothesis is true, much better results are known (for instance, see \cite{Ingham1937, Ramare2003, Dudek2015, Dudek2016, Carneiro2019}).

%A major improvement is due to Ingham \cite{Ingham1937} who showed that for some positive constant $c$, if $\zeta(\frac{1}{2} + \iota t) = O(t^c)$ then $p_{n+1}-p_n < p_n^\theta$ for any $\theta>\frac{1+4c}{2+4c}$. Here, O refers to the big O notation, $\zeta(.)$ denotes the Riemann zeta function and $\pi(.)$ the prime-counting function. Knowing that $c > \frac{1}{6}$ is admissible, one obtains that $\theta > \frac{5}{8}$. 

%\begin{theorem}
%For any $N > N_0$, there exists an $n\in [N+1,2N+1]$ such that $D(p_n) \approx g_n$ where $g_n = p_{n} - p_{n-1}$. 
%\end{theorem}
%\begin{proof}
%The alternate interpretation of the definition of degree of insulation gives $D(p_n) \approx \min(p_{n+1}-p_n, p_n-p_{n-1}) = (p_n - p_{n-1}) \times \min\left(1,~\frac{p_{n+1}-p_n}{p_n - p_{n-1}}\right)$ expressible as the ratio of consecutive prime gaps. From \citep{Pintz2015, Theorem 3.8} for $n\in [N+1,2N+1]$, we have
%\begin{equation*}
%\frac{p_{n+1}-p_n}{p_n - p_{n-1}} > c_0 \frac{\log{N} \log_2{N} \log_4{N}}{(\log_3{N})^2}
%\end{equation*}
%where $c_0$ is a constant and $\log_{t}{N}$ is the t-fold iterated logarithmic function. As $N$ is sufficiently large, so it turns out the right side of the inequality is greater than one. Thus, $D(p_n) \approx p_n - p_{n-1} = g_n$. 
%%$N \leq \exp\left(\frac{\log_3{4}}{c\log_2{3}}\right)$
%\end{proof}

\begin{theorem}
The condition for insulation of a prime $p_n$ is equivalent to
\begin{align*}
\max\{g_{k+1},\min\{g_{k-1},g_{k-2}+1\}\} < g_k < g_{k-1}+\max\{0, g_k+1-g_{k+1}\} ,
\end{align*}
where $g_n=p_{n+1}-p_n$ represents the gap between consecutive primes. 
\end{theorem}
\begin{proof}
The prime $p_n$ is insulated iff $D(p_k) > \max\{D(p_{k-1}),D(p_{k+1})\}$. It can be written as $D(p_k) - \max\{D(p_{k-1}),D(p_{k+1})\} > 0$ which is equivalent to
\begin{align*}
&\min\{D(p_k)-D(p_{k-1}),D(p_k)-D(p_{k+1})\} > 0 \\
&\min\{\min\{g_k-1, g_{k-1}\}-D(p_{k-1}),\min\{g_k-1, g_{k-1}\}-D(p_{k+1})\} > 0 \\
&\min\{g_k-1-D(p_{k-1}), g_{k-1}-D(p_{k-1}), g_k-1-D(p_{k+1}), g_{k-1}-D(p_{k-1})\} > 0
\end{align*}
where $D(p_k)$ is substituted in the second line using Theorem \ref{Dg}. The obtained inequality shows that the minimum of the four entries must be positive, which implies that every entry will be positive. Since
\begin{align*}
&g_k-1-D(p_{k-1}) = g_k-1-\min\{g_{k-1}-1, g_{k-2}\} = \max\{g_k-g_{k-1}, g_k-1-g_{k-2}\} , \\
&g_{k-1}-D(p_{k-1}) = g_{k-1}-\min\{g_{k-1}-1, g_{k-2}\} = \max\{1, g_{k-1}-g_{k-2}\} , \\
&g_k-1-D(p_{k+1}) = g_k-1-\min\{g_{k+1}-1, g_{k}\} = -\min\{1, g_{k+1}-g_{k}\} , \\
&g_{k-1}-D(p_{k+1}) = g_{k-1}-\min\{g_{k+1}-1, g_{k}\} = - \min\{g_{k}-g_{k-1}, g_{k+1}-g_{k-1}-1\} ,
\end{align*}
so, we get
\begin{align*}
&\max\{g_k-g_{k-1}, g_k-1-g_{k-2}\} > 0, \\
&\max\{1, g_{k-1}-g_{k-2}\} > 0, \\
&\min\{1, g_{k+1}-g_{k}\} < 0, \\
&\min\{g_{k}-g_{k-1}, g_{k+1}-g_{k-1}-1\} < 0
\end{align*}
which are the desired conditions. Notice that the second inequality $\max\{1, g_{k-1}-g_{k-2}\} > 0$ is trivially true because the presence of $1$ makes the condition independent of the value of $g_{k-1}-g_{k-2}$. Also note that the third condition $\min\{1, g_{k+1}-g_{k}\} < 0$ will be true if and only if $g_{k+1}-g_{k} < 0$. The first and fourth inequalities can be expressed as
\begin{align*}
&g_k-g_{k-1}+\max\{0, g_{k-1}-1-g_{k-2}\} > 0, \\
&g_{k}-g_{k-1}+\min\{0, g_{k+1}-1-g_k\} < 0
\end{align*}
respectively, and combined into a single condition:
\begin{align*}
\min\{0, g_{k-2}+1-g_{k-1}\} < g_k-g_{k-1} < \max\{0, g_k+1-g_{k+1}\} .
\end{align*}
So, we now have
\begin{align*}
&g_{k+1} < g_k, \\
&g_{k-1}+\min\{0, g_{k-2}+1-g_{k-1}\} < g_k < g_{k-1}+\max\{0, g_k+1-g_{k+1}\}
\end{align*}
that can be further combined into a single condition which is our final result.
\end{proof}

Despite the apparent similarity between degree of insulation and gap between primes, note that unlike gaps, the value of $D(p)$ can be odd as well.

\begin{theorem}
For a prime $p > 2$, if $D(p)$ is odd then $p+(D(p)+1)$ is prime, else if $D(p)$ is even then $p-D(p)$ is prime.
\end{theorem}
\begin{proof}
From Theorem \ref{Dg}, $D(p_k) = \min\{p_{k+1}-p_k-1, p_k-p_{k-1}\}$. Since the difference of two primes is always even, so, if $D(p_k)$ is odd then $D(p_k) = p_{k+1}-p_k-1 \Rightarrow p_{k+1} = p_k+D(p_k)+1$, else if $D(p_k)$ is even then $D(p_k) = p_{k}-p_{k-1} \Rightarrow p_{k-1} = p_k-D(p_k)$. 
\end{proof}

In order to carry out a formal study, let us first define:
\begin{equation}
f_k = \frac{\#\{p: D(p)=k ~\& ~ p\leq x \}}{\#\{p: p\leq x \}}
\end{equation}
which refers to the fraction of primes with $D(p)=k$ over all primes below $x$. Figure \ref{freq1} is a scatter plot which is evidently dense for smaller values of $D(p)$. This phenomenon is well captured in Figure \ref{freq2} which is the plot of $f_k$. The graph depicts that the actual values approximately lie on the Gaussian curve:
\begin{equation} \label{fk}
\frac{\#\{p: D(p)=k ~\& ~ p\leq x \}}{\#\{p: p\leq x \}} \approx \frac{1}{\sqrt{2\pi} \sigma(x)} \exp\left(-\frac{1}{2}\left(\frac{k-1}{\sigma(x)}\right)^2\right)
\end{equation}
where $\sigma$ is a parameter dependent on $x$.

\begin{figure}[h!]
\centering
\includegraphics[scale=0.6]{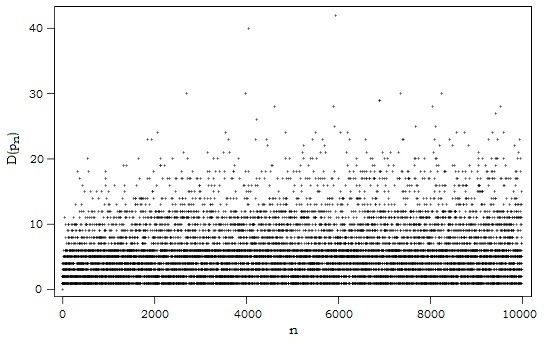}
\caption{Scatter plot of $D(p_n)$ for $n<10000$}
\label{freq1}
\end{figure}

\begin{figure}[h!]
\centering
\includegraphics[scale=0.41]{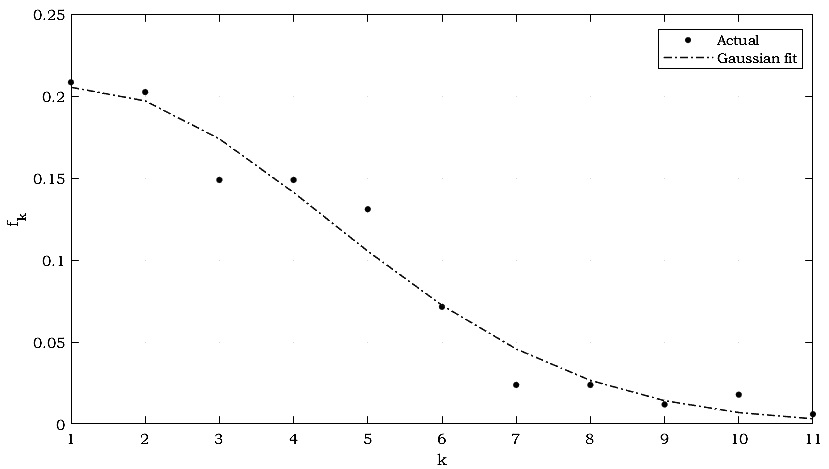}
\caption{Plot of $f_k$ with $1\leq k\leq 11$ for primes $p<1000$}
\label{freq2}
\end{figure}

%For primes up to $10^5$, the pattern is slightly deviated as evident from Figure \ref{freq2}. It is expected to behave nicely for larger range of primes.
%
%\begin{figure}[h!]
%\centering
%\includegraphics[scale=0.45]{Figure4.jpg}
%\caption{Plot of $f_k$ for $1\leq k\leq 30$ with primes $p<10^5$}
%\label{freq2}
%\end{figure}

\begin{theorem}
If Hardy-Littlewood conjecture on twin-primes is true, then
\begin{equation}
\#\{p: D(p)=k ~\& ~ p\leq x \} \sim \frac{2Cx}{(\log{x})^2} \exp\left(- 4 \pi C^2 \left(\frac{k-1}{\log{x}} \right)^2\right) ,
\end{equation}
where $C = \prod\limits_{p\in\mathbb{P}_{>2}}\left(1-\frac{1}{(p-1)^2}\right) \approx 0.6601618\dots$ is a constant. 
\end{theorem}
\begin{proof}
For $k=1$, we have
\begin{align*}
\frac{\#\{p: D(p)=1 ~\& ~ p\leq x \}}{\#\{p: p\leq x \}} \approx \frac{1}{\sqrt{2\pi} \sigma(x)} \exp\left(-\frac{1}{2}\left(\frac{1-1}{\sigma(x)}\right)^2\right) = \frac{1}{\sqrt{2\pi} \sigma(x)}  
\end{align*}
where the left-hand-side of the equation can be written as $\frac{\pi_2(x)}{\pi(x)}$ (where $\pi_2(x)$ is the number of twin-primes till $x$) since counting the number of primes with $D(p)=1$ gives the count of twin-prime pairs. Substitute the asymptotic formula for the prime counting function (long established) and twin-prime counting formula (conjectured by Hardy-Littlewood and heuristically verified \cite{Wolf2011}) in the earlier equation to obtain 
\begin{align*}
\frac{1}{\sigma(x)} \approx \frac{\pi_2(x)}{\pi(x)} \sqrt{2\pi} \sim \frac{2C \frac{x}{(\log{x})^2}}{\frac{x}{\log{x}}} \sqrt{2\pi} = \frac{2C\sqrt{2\pi}}{\log{x}} ,
\end{align*}
where $C \approx 0.6601618\dots$ is called the twin-prime constant. Substituting the expression for $\sigma(x)$ in (\ref{fk}) gives
\begin{align*}
\frac{\#\{p: D(p)=k ~\& ~ p\leq x \}}{\#\{p: p\leq x \}} &\sim \frac{2C}{\log{x}} \exp\left(- 4\pi C^2 \left(\frac{k-1}{\log{x}}\right)^2\right) 
\end{align*}
which is the final result.
\end{proof}

Let $\nu(x,g)$ be the number of primes up to $x$ that differ by gap $g$. Zhang in 2014 proved that there are infinitely many pairs of primes that differ by $g$ for some $g$ less than 70 million \cite{Zhang2014}; a bound which has been improved to 246 unconditionally (see \cite{Maynard2015, Polymath2014} and references therein). Now, if $\nu(x,2)\geq \nu(x,4)\geq \nu(x,6)\geq \ldots \geq \nu(x,g)$ for atleast up to $g=246$ then it would directly imply the infinitude of twin-primes. This is a topic for a separate future work.

%The approximation (1) is known ([10], improving on [12]) to hold in the sense of asymptotic equivalence for all $x\in [X,2X]$ provided that $X^{7/12} \leq y \leq X$, whilst the celebrated result of Maier [13] shows that (1) does not hold in the sense of asymptotic equivalence for all $x\in [X,2X]$ when $y$ is of size $(\log{X})^A$, for any fixed $A > 0$. He showed that for $y=(\log{X})^A$, there exists $x\in [X,2X]$ such that
%\begin{equation}
%\pi(x+y) - \pi(x) > (1+\delta_A) \frac{y}{\log{x}} .
%\end{equation}
%for a positive constant $\delta_A$ depending only on $A$ (and he also obtained similar results for intervals containing fewer than the average number of primes).
%\cite{Maynard2015}

\section{Growth pattern in insulated primes}

The natural procedure for finding $D(p)$ is the definition itself (illustrated earlier) involves computation of $\pi(x)$. Another approach is to determine the surrounding primes of the given prime (in view of Theorem \ref{Dg}), however, it might be inconvenient for extremely large primes. The number of computations of $\pi(x)$ can be reduced using the results proved in the last section and applying bracketing techniques for exact and faster numerical computation of $D(p)$. For instance, rather than beginning from $1$ and linearly searching, it is better to choose some $m_0$ as the starting point and then iterate to a better point using bisection method or metaheuristic algorithm. Using Proposition \ref{D<} or a sharper inequality, one can choose a better starting point which may allow faster algorithm for extremely large $p$.

%Another approach is numerical computation using Newton-Raphson (NR) method. This method requires an initial guess for $m$ to begin iterating and improving. Let $m_k$ denote the value of $m$ at $k$-th iteration, and $m_0$ is the initial guess. Generally, $m_0$ may be taken as one, but then it may get stuck in a local optima. In view of Corollary \ref{main}, it is more wise to take a high value which then reduces down to the optimal $m$. So, we initiate $m_0 = p$ itself. The second order convergence of NR method will ensure that no more than 3 to 5 iterations are required to reach near the closest integer of $m$. 

Insulated primes can be interpreted as the sequence of local maxima in the plot of degree of insulation. The sequence of insulated primes is 7, 13, 23, 37, 53, 67, 89, 103, 113, 131, 139, 157, 173, 181, 193, and so on. Using MATLAB command \texttt{cftool} for the Curve fitting toolbox, a variety of curves (with different settings) were tested which showed that $i_n$ obeys a power law. It was observed that the equation $y = 19.18 n^{1.093}$ is an extremely good fit as shown in Figure \ref{in10_5}.

\begin{figure}[h!]
\centering
\includegraphics[scale=0.4]{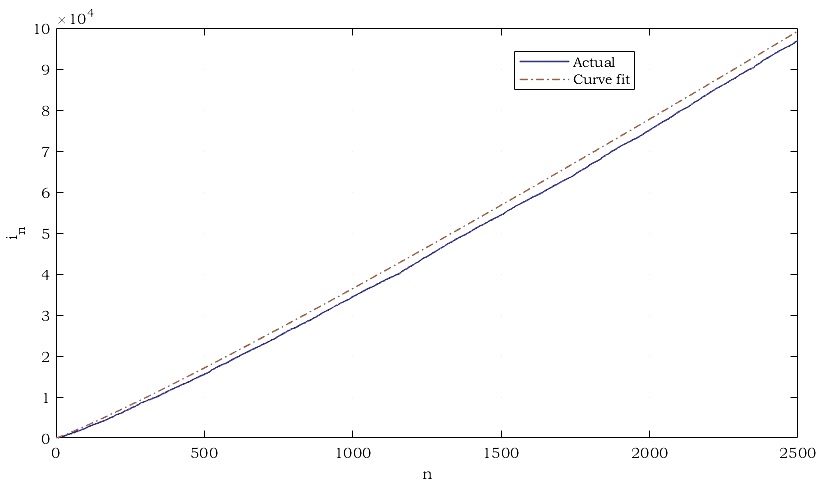}
\caption{Comparison of $i_n$ plots for primes less than $10^5$.}
\label{in10_5}
\end{figure}

The equation performs even better when tested for primes up to one million as shown in Figure \ref{in10_6}. It has an R-square value of almost unity which is extremely good. This allows to conclude that $i_n \sim 19.18 n^{1.093}$ is heuristically an accurate fit. This analysis is sufficient to convince that insulated primes are definitely well-behaved in comparison to primes or other prime subsets. Maybe they behave almost linearly for extremely large $n$, but such a result needs to be rigorously proven. %In fact for $i_n \sim a n^{b}$, this paper conjectures that $b\to 1$ as $n\to \infty$.

\begin{figure}[h!]
\centering
\includegraphics[scale=0.4]{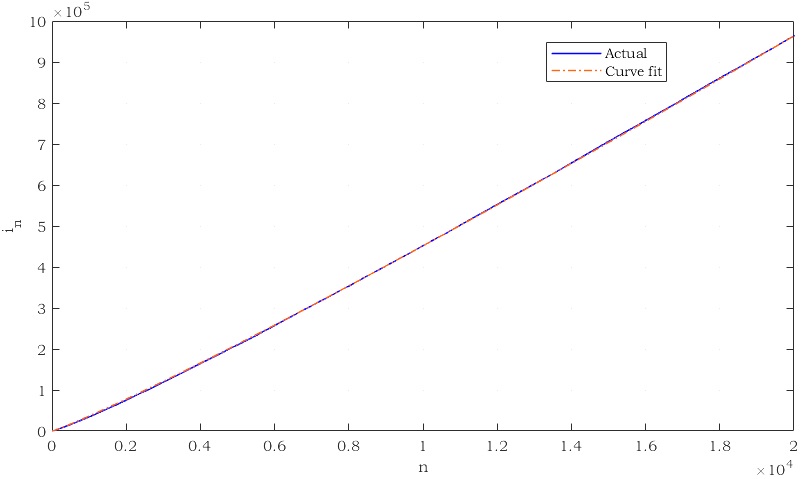}
\caption{Comparison of $i_n$ plots for primes less than $10^6$.}
\label{in10_6}
\end{figure}

\begin{figure}[h!]
\centering
\includegraphics[scale=0.35]{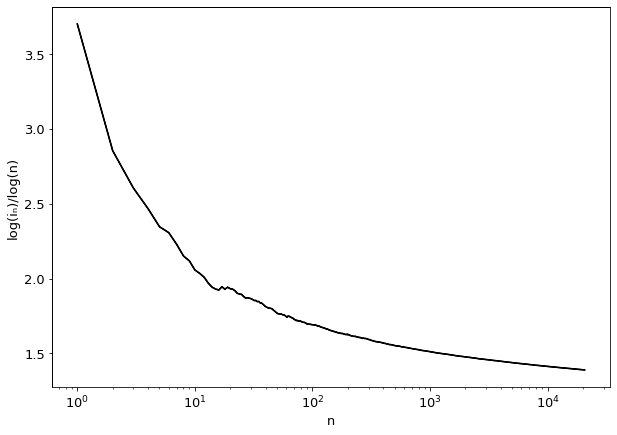}
\caption{Plot of $\frac{\log{i_n}}{\log{n}}$ versus $n$.}
\label{login}
\end{figure}

The idea is investigated here by evaluating $\frac{\log{i_n}}{\log{n}}$ for primes up to one million as shown in Figure \ref{login}. This analysis convincingly showed that the power (in the power law obeyed by $i_n$) will be constant for long intervals of $n$. Still the analytic exploration is open for interested reader.

%%%%%%%%%%%%%%%%%%%%%%%%%%%%%%%%%%%%%%%%%%%%%%%%%%%%%%%%%%%%%%%%
\section{Conclusion and Future scope}

This paper explores the properties of a new sequence called the insulated primes. They are defined using the concept of degree of insulation which has deep connection with gaps between consecutive primes. Mathematical analysis is conducted which unraveled several intriguing results. Finally, heuristics show that $n$-th insulated prime $i_n$ obeys a power law. The sequence of insulated primes is found to be well-behaved in comparison to the regular primes. 

An extension of the insulated primes is the highly insulated primes. Applying the definition of insulation on the set of primes $\mathbb{P}$ gave the set of insulated primes $\mathbb{I}$, likewise applying the definition of insulation on $\mathbb{I}$ gives the set of highly insulated primes $\mathbb{I}_H$ where $i_n$ is highly insulated iff $D(i_n) > \max\{D(i_{n-1}),D(i_{n+1})\}$. This is recorded as sequence A339188 in OEIS listed as 23, 53, 89, 211, 293, and so on. Interested reader may investigate this sequence. In addition to this, observing that application of the concept of insulation repeatedly reduces the size of the subsequent set, we wonder whether its possible that at some stage the resulting set becomes finite. So, it is worthwhile to ponder on the cardinality of a set obtained by repeated insulation of primes. 

%Good prime \cite{A028388}

\vskip 10pt
\noindent {\bf Availability of code.}
Mathematica codes of the sequences can be found in OEIS. Python code is made available at \url{https://github.com/anuraag-saxena/Insulated-Primes}.

%\vskip 10pt
%\noindent {\bf Acknowledgments.}
%The authors are grateful to the anonymous reviewer and the editor for the constructive comments.

\vskip 10pt
\noindent {\bf Conflict of Interests.}
The authors have no competing interests to disclose.


\begin{thebibliography}{28} \footnotesize

%\bibitem{Duart2018}
%P. Dusart, Explicit estimates of some functions over primes, {\it Ramanujan J.} {\bf 45} (2018), 227--251.

\bibitem{Alweiss2018}
R. Alweiss and S. Luo, Bounded gaps between primes in short intervals, {\it Res. Number Theory} {\bf 4} (2018), art. 15. %doi:10.1007/s40993-018-0109-y 

\bibitem{Baker2001}
R. C. Baker, G. Harman, and J. Pintz, The difference between consecutive primes, II, {\it Proc. Lond. Math. Soc.} {\bf 83(3)} (2001), 532--562. %doi:10.1112/plms/83.3.532.

\bibitem{Carneiro2019}
E. Carneiro, M. B. Milinovich, and K. Soundararajan, Fourier optimization and prime gaps, {\it Comment. Math. Helv.} {\bf 94} (2019), 533--568.

\bibitem{Dudek2015}
A. W. Dudek, On the Riemann hypothesis and the difference between primes, {\it Int. J. Number Theory} {\bf 11} (2015), 771--778.

\bibitem{Dudek2016}
A. W. Dudek, L. Greni{\'e}, and G. Molteni, Primes in explicit short intervals on RH, {\it Int. J. Number Theory} {\bf 12} (2016), 1391--1407.

\bibitem{Guy2004}
R. K. Guy, {\it Unsolved Problems in Number Theory}, Springer Verlag, New York (2004).

\bibitem{Heath-Brown1979}
D. R. Heath-Brown and H. Iwaniec, On the difference between consecutive primes, {\it Invent. Math.} {\bf 55(1)} (1979), 49--69.

\bibitem{Heath-Brown1988}
D. R. Heath-Brown, The number of primes in a short interval, {\it J. Reine Angew. Math.} {\bf 389} (1988), 22--63.

\bibitem{Hoheisel1930}
G. Hoheisel, Primzahlprobleme in der Analysis, {\it Sitz. Preuss. Akad. Wiss.} {\bf 33} (1930), 3--11.

\bibitem{Heilbronn1933}
H. A. Heilbronn, Über den Primzahlsatz von Herrn Hoheisel, {\it Math. Z.} {\bf 36(1)} (1933), 394--423. %doi:10.1007/BF01188631.

\bibitem{Huxley1972}
M. N. Huxley, On the difference between consecutive primes, {\it Invent. Math.} {\bf 15(2)} (1972), 164--170. %doi:10.1007/BF01418933.

\bibitem{Ingham1937}
A. E. Ingham, On the difference between consecutive primes, {\it Q. J. Math.} {\bf 8(1)} (1937), 255--266. %doi:10.1093/qmath/os-8.1.255.

\bibitem{Kumar2020}
A. Kumar and A. Saxena, Insulated primes, preprint arXiv:2011.14210 (2020).

\bibitem{Kumar2022}
A. Kumar, Prime floor and prime ceiling functions, {\it Integers} {\bf 22} (2022), \#A41.

\bibitem{Maynard2015}
J. Maynard, Small gaps between primes, {\it Ann. of Math.} {\bf 181} (2015), 383--413.

\bibitem{Polymath2014}
D. H. J. Polymath, Variants of the Selberg sieve, and bounded intervals containing many primes, {\it Res. Math. Sci.} {\bf 1} (2014), Article no. 12.

%\bibitem{Polymath_erratum}
%DHJ Polymath, Erratum to: Variants of the Selberg sieve, and bounded intervals containing many primes, {\it Research in the Mathematical Sciences} {\bf 2} (2015), Article no. 15.

\bibitem{Ramare2003}
O. Ramar{\'e} and Y. Saouter, Short effective intervals containing primes, {\it J. Number Theory} {\bf 98} (2003), 10--33.

\bibitem{Rosser1962}
J.~B. Rosser and L. Schoenfeld, Approximate formulas for some functions of prime numbers, {\it Illinois J. Math.} {\bf 6(1)} (1962), 64--94.

\bibitem{Chudakov1936}
N. G. Tchudakoff, On the difference between two neighboring prime numbers, {\it Mat. Sb.} {\bf 1} (1936), 799--814.

\bibitem{Wolf2011}
M. Wolf, Some heuristics on the gaps between consecutive primes, preprint arXiv:1102.0481 (2011).

\bibitem{Zhang2014}
Y. Zhang, Bounded gaps between primes, {\it Ann. of Math.} {\bf 179(3)} (2014), 1121--1174.

%\bibitem{Chen1966}
%J.~R. Chen, \emph{On the representation of a large even integer as the sum of a
%  prime and the product of at most two primes}, Kexue Tongbao \textbf{17}
%  (1966), 385–386.
%  
%\bibitem{Pintz2015}
%J{\'a}nos Pintz, \emph{On the ratio of consecutive gaps between primes},
%  Analytic number theory (C.~Pomerance and M.~Rassias, eds.), Springer, 2015,
%  pp.~285--304.
%
%\bibitem{Sondow2009}
%J.~Sondow, \emph{Ramanujan primes and Bertrand's postulate}, Amer. Math.
%  Monthly \textbf{116} (2009), 630–635.

\end{thebibliography}
\end{document}